\documentclass[12pt,reqno]{amsart}

\usepackage{amsmath,amssymb,enumerate}
\usepackage{hyperref}
\newtheorem{theorem}{Theorem}
\newtheorem{lemma}{Lemma}
\newtheorem{conj}{Conjecture}

\newcommand{\eps}{\varepsilon}

\newcommand{\zeile}{\vspace{\baselineskip}}

\newcommand{\N}{\mathbb{N}}

\newcommand{\Z}{\mathbb{Z}}

\newcommand{\C}{\mathbb{C}}

\allowdisplaybreaks

\begin{document}

\title{Asymptotics of Goldbach representations}
\author{Gautami Bhowmik}
\address{G. Bhowmik: Laboratoire Paul Painlev{\'e}, Labex CEMPI, Universit{\'e} Lille, 59655
Villeneuve d'Ascq Cedex, France}
\email{Gautami.Bhowmik@math.univ-lille1.fr}

\author{Karin Halupczok}
\address{K. Halupczok: Mathematisch-Naturwissenschaftliche Fakult{\"a}t
Hein\-rich-Heine-Universit{\"a}t D{\"u}sseldorf, Universit{\"a}tsstr. 1,
40225 D{\"u}ssel\-dorf, Germany}
\email{karin.halupczok@uni-duesseldorf.de}
\maketitle

\hfill {\it Dedicated to Kohji Matsumoto}

\begin{abstract} We present a historical account of the asymptotics of classical Goldbach representations with 
special reference to the equivalence with the Riemann Hypothesis. When the primes are chosen from an arithmetic progression  comparable but weaker relationships exist with the zeros of L-functions. 
\end{abstract}

\section{Introduction}

One of the oldest open problems today, known as the Goldbach conjecture, is to know if every even integer greater than 2 can be expressed 
as the sum of two prime numbers.
It is known since very long that the conjecture
is statistically true and it is empirically supported by calculations  for all numbers the threshold of which has gone up from   
$10,000$ in 1855 \cite{Desboves} to $4\times 10^{18}$ \cite{OHP} in 2014.  
In Section~\ref{sec:HaLi} we give some historical results that  support the Goldbach  conjecture.

Though the conjecture in its totality seems out of reach at the moment it generates 
a lot of mathematical activity. What follows, for the most part {\it expository}, concerns only a few of these aspects while  many are obviously omitted.

Instead of studying directly the Goldbach function $g(n)=\sum_{p_{1}+p_{2}=n}1$
which counts the number of representations of an integer $n$
as the sum of two primes $p_{1}$ and $p_{2}$ and is expected to be non-zero for even $n>2$, it is convenient to treat  a smoother version using logarithms.
This is easier from the point of view of analysis and the preferred function here is  the {\it weighted } Goldbach function
\[
    G(n)=\sum_{m_{1}+m_{2}=n}\Lambda(m_{1})\Lambda(m_{2})
\]

where  $\Lambda$ denotes von Mangoldt's function
\[
   \Lambda(n)=\begin{cases}
       \log p, &\text{ if } n=p^{k} \text{ for some prime } p
   \text{ and } k\in\N\\ 0, &\text{ otherwise}
  \end{cases}
\]
so that  $g(n)$ can be recovered from $G(n)$ by the use of partial summation
and  the last being sufficiently large, more precisely $G(n) > C\sqrt n$, would imply the Goldbach conjecture. 
As is common in analytic number theory we study the easier question  of the average order of the Goldbach functions
where the first results are at least as old as Landau's.

The dominant term in these results can be obtained easily but the oscillatory term involves infinitely many nontrivial zeros of the Riemann zeta function
 and any good asymptotic result involving upper bounds on the error term is conditional to the Riemann Hypothesis (RH). The lower bounds are
however unconditional. We indicate some such average results in Section~\ref{sec:mean}.
 
 Interestingly obtaining a good average order is actually equivalent to the Riemann Hypothesis.  
 We include the proof  of the last statement in Section 
~\ref{app:RHGo} since it has to be extracted from scattered parts in other papers.
 
 A variation of the classical Goldbach problem is one where the summands are 
primes  in arithmetic progressions.  We present some information on the exceptional set in 
Section~\ref{sec:ExcSetsAP}
and a proof in Section
~\ref{app:lem1}.  In this context average orders with good error terms are  necessarily conditional to the appropriate 
Generalised Riemann Hypothesis (GRH). Equivalences with such hypotheses till
now exist only in special cases (Theorem \ref{th:BHMS2}) 
and seem difficult in general because of the possible Siegel zeros. Some of this is explained in Section~\ref{sec:SZe}.

\section{Goldbach Conjecture is often true}
\label{sec:HaLi}

\subsection{Hardy--Littlewood conjecture}
\label{sec:HL}

Hardy and Littlewood \cite{HL} pursued the ideas of Hardy--Ramanujan\cite {HR}
and expected an asymptotic
formula  to hold for $G(n)$, which is,
\begin{conj}[Hardy--Littlewood]
\label{conj:HL}
  The approximation $G(n)\sim J(n)$ holds for even $n$ with 
  \[
  J(n):= n C_{2} \prod_{\substack{p\mid n \\ p>2}} \frac{p-1}{p-2}
  \]
and $C_{2}:=2\prod_{p>2}\Big(1-\frac{1}{(p-1)^{2}}\Big)$.
\end{conj}

The above $C_2$  is known as the twin prime constant and is approximately 1.32.
Our interest is largely in the error term
\[
  F(n):= G(n)-J(n)
\]
of the above conjecture, 
where $J(n):=0$ for odd $n$. 
What can be said about $F(n)$? How can we estimate it to show that it
is indeed small? Where are the sign changes?

To tackle these questions, it is natural to consider the average
$\sum_{n\leq x} F(n)$ or the second moment $\sum_{n\leq x} F(n)^{2}$.
Motivated by Conjecture \ref{conj:HL}, we should expect  
$o(x^{3})$ for the second moment and $o(x^{2})$ for the average.

\subsection{Exceptional Sets}
\label{sec:ExcAP}

Using what is now called the  Hardy--Little\-wood--Ramanujan circle method, and  assuming the GRH, Hardy and Littlewood proved \cite{HL}  the estimate
\[
   \sum_{n\leq x} F(n)^{2}\ll x^{5/2+\eps}
\]
for the second moment.

From these nontrivial bounds for the second moment
we are able to deduce estimates for the number of exceptions to the representation as the sum of two primes
in the
following way. 

Let $E(x)=\#\{n\le x, n\in2\N; \ n\neq p_{1}+p_{2}\}$
denote the size of the exceptional class depending on $x$, i.e.\ the number of even integers up to $x$ which do not satisfy the Goldbach conjecture.
For these exceptions, we have $|F(n)|\geq c n$ for some constant
$c>0$, since $J(n)\geq C_{2}n$. Therefore
\[
   x^{2}E(x) \ll 
   \sum_{\substack{n\leq x\\ 2\mid n\\ n \neq p_{1}+p_{2}}} n^{2} \ll
   \sum_{n\leq x} |F(n)|^{2} \ll x^{5/2+\eps},
\]
so  after a dyadic dissection $$E(x)\ll x^{1/2+\eps}$$ under the GRH.

More than 60 years later  Goldston \cite{Gol}  improved the Hardy--Littlewood bound to
\begin{equation}\label{eq:GBGol}
E(x)\ll x^{1/2}\log^{3} x
\end{equation}
still under the assumption of GRH.

By now many authors have succeeded in obtaining nontrivial bounds for $\sum_{n\leq
  x}F(n)^{2}$ unconditionally and we know that the
Hardy--Littlewood conjecture is true {\it on average}. 

Just after 1937 when  Vinogradov's method \cite{Vino} became available, Van der Corput, Chudakov and Estermann  
\cite{vdc,chu,est} independently obtained the first unconditional estimate of the type $E(x)=o(x)$. More precisely,
they proved that
\[\sum_{n\leq x}F(n)^{2}\ll
x^{3} \log ^{-A}x\]
so that 
\begin{equation}
\label{eq:cce}
E(x)\ll 
x\log^{-A} x
\end{equation}
for any $A>0$.

Thus we know that the Goldbach conjecture is true {\it in a statistical sense}.

In 1975 Montgomery and Vaughan \cite{MV}, improved this by using an effective form of Gallagher's work on distribution of zeros 
of $L$-functions and showed
that there exists a positive effectively computable constant $\delta>0$ such that, 
for all large $x>x_0(\delta)$, $$E(x)\le x^{1-\delta}.$$

The best published proof as of now is 
with $\delta=0.121$, i.e.\ 
\begin{equation}
\label{eq:lu}
E(x)\ll x^{0.879}
\end{equation}
by Lu \cite{Lu} who obtained this by a variation of the circle method.
In the very recent preprint \cite{pin} of Pintz  the bound
\begin{equation}
  \label{eq:pin}
  E(x)\ll x^{0.72}
\end{equation}
is achieved, i.e.\ $\delta=0.28$.

Further information can be found in the survey article \cite{VauOver}
of Vaughan.

\subsection{Average Orders}\label{sec:mean}
As early as 1900 an asymptote for the average order of $g(n)$ was known
due to Landau \cite{Land} who showed that

\[
\sum_{n\le x} g(n)\sim \frac{1}{2}\frac{x^2}{\log^{2} x}.
\]
His result agreed with the conjecture that $g(n)$ should be
approximated by $\frac{J(n)}{\log^2 n}$.

After almost a century  Fujii \cite{fuj,fuj2} studied the  oscillating term. He first obtained, assuming the RH, that
\[\sum_{n\leq x}F(n)=O(x^{3/2})
\]
using the work of Gallagher of 1989.

Fujii could then extract  an error term smaller than the main oscillating term
  \cite{fuj2}, i.e.\
\begin{theorem}[Fujii's theorem]
 Assume RH. Then, for $x$ sufficiently large, we have
   \[
       \sum_{n\leq x}F(n) =-4x^{3/2}\Re\Big(\sum_{\gamma>0}
       \frac{x^{i\gamma}}{(1/2+i\gamma)(3/2+i\gamma)}\Big) 
      + O( (x\log x)^{4/3})
   \]
   where $\gamma$ denotes the imaginary parts of the zeros of the Riemann zeta function.
\end{theorem}

In 2007-8 Granville \cite{Gra} also studied the average
$\sum_{n\leq x} F(n)$ and obtained the same error term.

It was conjectured by Egami and Matsumoto that the error term would be $O(x^{1+\epsilon})$ for $\epsilon$ positive
and this  was reached by Bhowmik and Schlage-Puchta \cite{B_SP} using the distribution of primes in short intervals to estimate exponential sums 
close to 0. More precisely, under the assumption of the RH,  the asymptotic result  therein can be stated in the form 
$$
\sum_{n\le x}G(n)=\frac{1}{2}x^2 + H(x) +O(x\log^5 x)
$$
where $H(x)=-2\sum_{\rho}\frac{x^{\rho +1}}{\rho(\rho +1)}$ is the oscillating term involving $\rho$, the non-trivial zeros of the Riemann zeta 
function.

Languasco and Zaccagnini \cite{LaZa} used the circle method to improve the power of the logarithm in the error term from $5$ to $3$.

Goldston and Yang \cite{GY} could also reach $O(x\log^{3} x)$ in 2017 using the method of 
\cite{B_SP}.

\subsection{Equivalence with RH}

While it is now clear that error terms for the average of $G(n)$ depends on the zeros of the Riemann zeta function  $\zeta(s)$ and hence 
is meaningful only assuming
a suitable hypothesis, it is interesting to see whether we can obtain information on the zeros of $\zeta(s)$ if we have an asymptotic expansion.
This was investigated by Granville in 2007 who  stated  that (Theorem 1A  in \cite{Gra}) 
\begin{theorem}
\label{th:AG1A}
The  RH is equivalent to the estimate $$\sum_{n\leq x} F(n) \ll x^{3/2+o(1)}.$$
\end{theorem}

 The sketch of proof in \cite{Gra} is not sufficient to obtain the RH from the above asymptote.
 However in the recent paper  by Bhowmik, Halup\-czok, Matsumoto and Suzuki \cite{BHMS}
we were able to reconstruct part of the proof
of this theorem and it was completed with Ruzsa in \cite{BhRu}. Some details are given in Section~\ref{app:RHGo}.

\subsection {Omega-results}

The lower bounds arise while showing that there exist $n$ for which 
$G(n)$ is large. Since the two main terms in the asymptotic expansion of
the average are continuous this already contributes to the error term.
It is to be noted that the lower bounds obtained are uncondtional unlike the upper ones.

The first observation seems to be that of Prachar back in 1954 \cite{Pra}
when he proved that there are infinitely many integers $n$ such that
$$ g(n)> C\frac{n}{\log^2 n} \log\log n.$$
His study was on the lines of  Erd\H os earlier for the number of solutions of $n$ as the sum of 2 prime squares with combinatorial arguments for the number of primes in residue classes modulo an integer with many prime factors.

Half a century later  Giordano \cite{Gio}  studied the irregularity of $g(n)$ 
depending on whether or not $n$ is divisible by many small primes by using 
the Prime Number Theorem for arithmetic progression.

Having overlooked these earlier results the authors of \cite{B_SP} again showed that the error term in $\sum_{n\le x}G(n)$ is $ \Omega (x \log \log x)$ 
by considering the exceptional bounded modulus for which a Siegel zero might exist and using Gallagher's density estimates.

Another way to find an omega term is to study the natural boundary of the generating function $\sum_n\frac{G(n)}{n^s}$ as was mentioned in \cite{B_SP11}.

\section{Goldbach representations
in arithmetic progressions}

As a restricted form  of the original problem one considers the possibility of representing every even number as the sum
of two primes in a given residue class.
Here again the conjecture is known to be almost always true. The exceptional set
$$E(x;q,a,b)=\#\{n\le x, n\in2\N; \ n\neq p_{1}+p_{2}; \ p_1\equiv a(q), p_2\equiv b(q)\}$$
is shown, for example in \cite{LZ}, to satisfy, for an effectively  computable positve $\delta$
$$ E(x;q,a,b)\ll \frac{x^{1-\delta}}{\phi(q)}$$
for all $q \le x^{\delta}$.

\subsection{Exceptional Sets}
\label{sec:ExcSetsAP}

We can apply the results of the estimation of the exceptional set of the unrestricted Goldbach representation
to obtain similar results for an arithmetic progression
with residue $h$ mod $q$. 
 Similar to $E(x;q,a,b)$, we define
\[
  E_{h,q}(x):=\#\{n\le x, n\in2\N; \ n\neq p_{1}+p_{2}; \ n\equiv h(q)\}.
\]
We show that  elementary methods already suffice to deduce
 nontrivial
estimates for the number of exceptions in an arithmetic progression
from available nontrivial estimates for $E(x)$.

Thus a very simple deduction from \eqref{eq:cce}  gives
the statement that for \emph{almost all residues} $h$ mod $q$ we have
\begin{equation}\label{trivial_hq}
    E_{h,q}(x)\ll \frac{x}{q \log ^{C} x}
\end{equation}
arising from the fact that $\sum_{0\leq h<q} E_{h,q}(x)=E(x)$ so that if
\[
   H_{q}:=\{0\leq h<q;\ E_{h,q}(x)>xq^{-1}\log ^{-C} x\}
\]
denotes the set of exceptional residues mod $q$, we have
\[
    \# H_{q} \frac{x}{q \log ^{C} x} < \sum_{h\in H_{q}}
    E_{h,q}(x) \leq E(x) \ll x(\log ^{-2C} x),
\]
and hence $\# H_{q}\ll q \log ^{-C} x $.

We remark that  the estimate (\ref{trivial_hq}) above,
for all $h$ mod $q$, is already nontrivial for modulii $q$ with
$q\gg \log ^{C} x$ for any $C>0$ which is equivalent to the condition
\[
   \frac{x}{q\log ^{C} x} \ll \frac{x}{\log ^{2C}x}.\]
 
Let us now introduce a handy notation. If $B\subseteq
\N$, we write  

$B(x):=\#\{n\leq x;\ n\in B\}$ and $B_{h,q}(x):=\#\{n\leq x;\ n\in
B,\ n\equiv~ h\;(q)\}$.

\begin{lemma}[Number of exceptions in progressions]
\label{lem:l}
  Let $B\subseteq \N$ be such that 
  $B(x) \ll x \log ^{-A} x$ for any
  $A>0$. Then for all $C>0$ and for almost all $q\leq x^{1/2}$ we have
  $B_{h,q}(x) \ll xq^{-1}\log ^{-C}x$, 
uniformly in all residues $h$ mod $q$.
\end{lemma}

A proof is given in 
Section~\ref{app:lem1}.

 From this Lemma and \eqref{eq:cce} we know that for any $C>0$, and for almost all $q\leq x^{1/2}$
we have 
 $E_{h,q}(x)\ll xq^{-1}\log ^{-C}x$ unconditionally
for all $h$ mod $q$.

It follows immediately that for \emph  {almost all} modulii $q$ with
\\linebreak
$x^{1/2}\log ^{D}x\ll q\ll x^{1/2}(\log x)^{D}$, 
we have 
$$E_{h,q}(x)\ll x^{1/2}\log ^{C}x$$ 
for
all $h$ mod $q$, a  bound  known \emph {for all} modulii
under the assumption of the Riemann hypothesis.

Assuming RH, so that \eqref{eq:GBGol} is available,
we conclude  as in the proof of Lemma \ref{lem:l},
 that for almost all $q$ with 
  $x^{1/2}\ll q\leq x^{1/2}$
  we have $$E_{h,q}(x)\ll x^{7/8+\eps}/q\ll x^{3/8+\eps}$$ for all $h$
  mod $q$. 
(Since by \eqref{eq:ea2}, for the number $\#\mathcal{M}_{Q}$ of
exceptions in question,
$\#\mathcal{M}_{Q}\cdot x^{7/8+\eps}/Q\ll x^{3/4}E(x)^{1/4}(\log
x)\ll x^{7/8}\log^{2} x$, so $\#\mathcal{M}_{Q}=o(Q)$.)

Working with the best published unconditional bound \eqref{eq:lu}
 we can likewise deduce that
$$E_{h,q}(x)\ll x^{0.97+\eps}/q\ll x^{0.47+\eps}$$ for all $h$
  mod $q$ and almost all  $q$ with 
  $x^{1/2}\ll q\leq x^{1/2}$.

If we were working with the latest available bound \eqref{eq:pin} instead, 
we would deduce
\[ E_{h,q}(x)\ll x^{0.93+\eps}/q\ll x^{0.43+\eps}
\]
for all $h$ mod $q$ and almost all  $q$ with $x^{1/2}\ll q\leq x^{1/2}$.

%
\subsection{Mean Value}

In \cite{Gra}, Granville studied  the mean value of the Goldbach
representation number $G(n)$ in an arithmetic progression,
that is the sum
\begin{equation}
\label{eq:GBap}
   \sum_{\substack{n\leq x\\n\equiv c\:(q)}} G(n),
\end{equation}
in particular for the fixed residues $c=2$ and $c=0$ and 
stated
 the estimate
$\sum_{n\leq x, n\equiv 2\:(q)} F(n)\ll x^{3/2+o(1)}$.
We introduced
 the technically useful sum \cite{BHMS}
\[
S(x;q,a,b)= \sum_{n\leq x} \sum_{\substack{\ell+m=n\\\ell\equiv a,
    m\equiv b (q)}} \Lambda(\ell) \Lambda(m)
\]
a special case of which, $\sum_{a\:(q)} S(x;q,a,c-a)$ is  the above \eqref{eq:GBap},

and proved that, for $(ab,q)=1$, 
\[
S(x;q,a,b)= \frac{x^2}{2\phi(q)^2} +O(x^{1+B_q})
\]
where $B_q$ depends on the non-trivial zeros of the associated Dirichlet $L$-functions.
The oscillating term was also extracted but we do not discuss it here.


\subsection{Equivalence with  GRH-DZC}

Parallel to the equivalence of the RH and the asymptotic expansion of the classical Goldbach average,
it was believed that there is an equivalence between the RH for Dirichlet
$L$-functions $L(x,\chi)$ over all characters $\chi$ mod $m$ 
which are odd squarefree divisors of $q$ and the estimate
$\sum_{n\leq x, n\equiv 2\:(q)} F(n)\ll x^{3/2+o(1)}$   (\cite{Gra}, Thm.1B).

In \cite{BHMS} the same question was studied and the above claim  was recovered 
though partially, i.e.\ 
only in the case $b=a$, and additionally under the assumption of one more conjecture
on the non-trivial zeros of $L$-functions, called the Distinct Zero Conjecture 
(DZC), more precisely,
\begin{quote}
\textit{For any $q\geq 1$, any two distinct Dirichlet $L$-functions 
associated with characters of modulus $q$ 
do not have a common non-trivial zero, except for a possible 
multiple zero at $s=1/2$.}
\end{quote}

The equivalence obtained till now does not cover all residues $a$ and $b$.

\begin{theorem}[{from \cite[Thm.1]{BHMS}}]
\label{th:BHMS1}
Let {\upshape DZC} be true, $q$ be odd and $(a,q)=~1$.
Then, for any $\varepsilon>0$, the asymptotic formula
\begin{align}
\label{MainTheorem-1-formula}
S(x;q,a,a)=\frac{x^2}{2\varphi(q)^2}+O_{q,\varepsilon}(x^{3/2+\varepsilon})
\end{align}
is equivalent to the GRH for the functions $L(s,\chi)$ with any character $\chi$ mod $q$.
\end{theorem}
Continuing in the direction of possible equivalences  
a bit more can be said using the notation
$B_{\chi}=\sup\{\Re \rho_{\chi}\}$ and $B_q=\sup\{B_{\chi}\mid\chi\:(q)\}$
where $\rho_{\chi}$ are the non-trivial zeros of $L(s, \chi)$.

\begin{theorem}[{from \cite[Thm.3]{BHMS}}]
\label{th:BHMS2}
  Let $q,c$ be integers with $(2,q)\mid c$.

Assuming the GRH, we have
\begin{equation}
\label{thmsiegel_asymp}
\sum_{\substack{n\leq x\\n\equiv c\,(q)}}F(n)\ll x^{3/2}.
\end{equation}

On the other hand if we assume that
\begin{equation}
\label{eq:thmsiegel}
\sum_{\substack{n\leq x\\n\equiv c\:(q)}}F(n)\ll_{q,\varepsilon}
x^{3/2+\varepsilon}
\end{equation}
holds for any $\varepsilon>0$ and
that there exists a zero $\rho_0$ of $\prod_{\chi\:(q)}L(s,\chi)$ such that
\begin{enumerate}[(a)]
\item $B_q=\Re\rho_0$
\item $\rho_0$ belongs to a unique character $\chi_1\ (q)$
\item the conductor $q^\ast$ of $\chi_1\ (q)$ is squarefree and
  satisfies $(c,q^\ast)=1$,
\end{enumerate}
then $B_q=\Re \rho_0\le 1/2$.
\end{theorem}
Thus the GRH can be deduced only under several additional assumptions (a),(b) and (c),
from the asymptotic expansion \eqref{eq:thmsiegel}.

It is worth mentioning that  from  \eqref{thmsiegel_asymp} above, it
is immediate that \linebreak
$\sum_{n\leq x}F(n)\ll x^{3/2}$ when $q=1$ which in turn implies that under the RH
 we get the bound $E(x)\ll x^{1/2}\log x$. The last is an  improvement on Goldston's result \eqref{eq:GBGol}.


\subsection{Bombieri--Vinogradov Theorem and Siegel Zeros}
\label{sec:SZe}

Under the assumption that the
GRH can be deduced
for all $L(s,\chi)$ for any $\chi$ mod $q$ from asymptotic expansions 
we would have  a very short
proof of Bombieri--Vinogadov's theorem. Let us dwell on this possibility.

Let $D_{q}(x):=\sum_{\substack{n\leq x \\ n\equiv 2(q)}} (G(n)-J(n))$
and $\Delta_{q}(x):= \max_{a(q)}^{*}
|\psi(x;q,a)-\frac{x}{\varphi(q)}|$,
where trivially $\Delta_{q}(x)\ll xQ^{-1}\log ^{2}x$.
Let
\[
   E:=\{q\in\N;\ Q<q<2Q,\ |D_{q}(x)|\geq x^{3/2+\delta} \text{ for some
   }\delta>0 \},
\]
so that $q\not\in E \implies D_{q}(x)\ll x^{3/2+o(1)}$. From our assumption
this would imply GRH for $L(s,\chi)$ with any $\chi$ mod $q$.
Then $\Delta_{q}(x)\ll x^{1/2}\log ^{2}x$.

Hence
\[
   \# E\, x^{3/2+\delta} \leq \sum_{q\in E} |D_{q}(x)|,
\]
and by Lemma \ref{lem:bv} below we would have
\[
\sum_{q\in E} |D_{q}(x)| \leq \Big( \sum_{n\leq x} |G(n)-J(n)|^{2}\Big)^{1/2}
x^{1/2} \log ^{3/2}x\ll x^{2} \log ^{7/2} x
\]
by trivially estimating the last sum as $x^{3}\log ^{4}x$.
From this, we could deduce that $\#E \ll x^{1/2-\delta}\log ^{4} x $.

For the left hand side of the Bombieri--Vinogradov's theorem this yields
\begin{align*}
  \sum_{Q<q\leq 2Q} \Delta_{q}(x) 
  &=\sum_{q\in E} \Delta_{q}(x) + \sum_{q\not\in E} \Delta_{q}(x) \\
  &\ll x^{1/2-\delta} xQ^{-1}\log ^{6}x + Qx^{1/2}\log ^{2}x \ll
  x\log ^{-A}x
\end{align*}
if we assume that $x^{1/2-\delta}\log ^{6+A}x\ll Q\ll x^{1/2}\log
^{-2-A}x$. 

The elementary lemma used just above is  :
\begin{lemma}
  \label{lem:bv}
  For any sequence $(v_{n})_{n\in \N}$ of complex numbers 
and any integer $a$ with $a<x$ we have
\[
   \sum_{Q<q\leq 2Q} \Big|\sum_{\substack{n\leq x\\n\equiv a(q)}} v_{n}\Big|
   \leq \Big( \sum_{n\leq x} |v_{n}|^{2}\Big)^{1/2} x^{1/2} \log ^{3/2}x.
\]
\end{lemma}

\begin{proof}
  The left hand side is equal to the sum $\sum_{Q<q\leq 2Q} |\langle
  v,\varphi_{q}\rangle|$ with $\varphi_{q}(n)=1$ if $n\equiv a$ mod
  $q$ and $\varphi_{q}(n)=0$ otherwise. By Halasz--Montgomery's
  inequality, this is 
\[
   \leq \Big(\sum_{n\leq x} |v_{n}|^{2}\Big)^{1/2} \Big(\sum_{q_{1},q_{2}} \langle
   \varphi_{q_{1}},\varphi_{q_{2}}\rangle \Big)^{1/2}
\]
with
\[
\langle\varphi_{q_{1}},\varphi_{q_{2}}\rangle = \sum_{n\leq x}
\sum_{q_{1}\mid n-a} \sum_{q_{2}\mid n-a} 1 = \sum_{s, 0<x-a\leq s\leq
2x-a } \sum_{q_{1}\mid s} \sum_{q_{2}\mid s} 1 \ll x \log ^{3}x.
\]
\end{proof}


Concerning the Siegel zeros,  Fei in \cite{fei} has studied a similar question.
Assuming 
a certain version of  
\emph{weak} Goldbach conjecture, namely 
if for all even $n>2$,
\[
  G(n)\geq \frac{\delta n}{\log^{2} n},
\]
then the possible Siegel zero $\beta$ for $\chi$ mod $q$, where 
$q$ is a prime, $q\equiv 3$ mod $4$ and satisfies 
\[
  \beta\leq 1-\frac{c}{\log^{2} q}
\]
for some constant $c>0$.
Thus here we get a repulsion of the Siegel zero from the line $\Re s=1$ while 
assuming \eqref{thmsiegel_asymp} 
 on average 
we deduce $B_{q}\leq 1/2$ under (a), (b) and (c), which is a
repulsion of all the nontrivial zeros of $L$-functions mod $q$
from the line $\Re s=1$.

\newcommand{\du}{\,\textup{d}u}
\section{
Proof of equivalence of RH and 
Goldbach average } 
\label{app:RHGo}

In this section, we give an overview of the proof of Theorem
\ref{th:AG1A}, that the Riemann Hypothesis  is equivalent to the estimate 
\begin{equation}
\label{eq:f}
\sum_{n\leq x} (G(n)-J(n)) \ll_{\eps} x^{3/2+\eps}
\end{equation}
for any $\eps>0$.

Since $\sum_{n\leq x}J(n)-x^{2}/2\ll x \log x$,
we can write  \eqref{eq:f} equivalently  as
\begin{equation}
  \label{eq:ff}
  \sum_{n\leq x} G(n) =\frac{x^{2}}{2}+O(x^{3/2+\eps})
\end{equation}
for any $\eps>0$.

We concentrate on the proof of the deduction of the RH from
 \eqref{eq:f} 
that is to obtain $B=1/2$ for
$B:=\sup \{\Re \rho;\ \zeta(\rho)=0\}$.

Step 1. Let
$S(x):=\sum_{n\leq x} G(n)$ be the summatory function of $G(n)$.
A key issue is the proof of the asymptotic formula
\[
   S(x)=x^{2}/2+\sum_{\rho} r(\rho)
   \frac{x^{\rho+1}}{\rho+1} + E(x)
\]
with $E(x)\ll x^{2B}\log ^{5} 2x$
and $r(\rho):=-2/\rho$
( Theorem 2 in \cite{BHMS}). This involves a careful
transformation of the problem to an exponential sum setting where
Gallagher's lemma can be used, confer \cite[Lemma 9]{BHMS}.

Step 2. 
We define the Goldbach generating Dirichlet series as
\[
F(s)=\sum_{n=1}^{\infty}\frac{G(n)}{n^s}.
\]
This can be computed by using $S(x)$ in the integral
\[
   F(s)=s\int_{1}^{\infty} S(u)u^{-s-1} \du,
\]
where inserting the formula from Step 1 yields
\begin{align}
  F(s)
&=\frac{s}{2(s-2)}+\sum_{\rho}\frac{r(\rho)s}{(\rho+1)(s-\rho-1)}
+s\int_1^\infty E(u)u^{-s-1} \du \notag\\
&=\frac{1}{s-2}+\sum_{\rho}\frac{r(\rho)}{s-\rho-1}
+s\int_1^\infty E(u)u^{-s-1}\du+C_1, \label{eq:fs}
\end{align}
with
\[
C_{1}=\frac{1}{2}+\sum_{\rho}  \frac{r(\rho)}{\rho+1}
\]
and $r(\rho)=-2/\rho$.

From the above we can read off that for $\sigma>2$, the function $F(s)$
converges absolutely and is {\it analytic}.

Moreover $F(s)$ can be {\em continued meromorphically}
to the half plane $\sigma>2B$ since $E(u)\ll u^{2B}\log^{5} (2u)$.

Step 3. Assume $B<1$. Then we have
\begin{equation}
\label{eq:FB}
 1+B=\inf\{\sigma_0\ge\frac{3}{2} \mid
F(s)-\frac{1}{s-2}\text{ is analytic on }
  \sigma>\sigma_0\}.
\end{equation}
 Step 2 shows that the infimum on the right is at most $2B\leq B+1$, 

For the inequality in the other sense we observe that this is trivially true for $B=1/2$,
so we may assume that $1/2<B<1$.  

Now $\max(2B,3/2)<1+B$ being a strict inequality, there exists a
real number $\eps>0$ such that $\max(2B,3/2)<1+B-\eps$ holds true.
Then by the definition of $B$, there exists a zero $\rho$ such that
$1/2<B-\eps<\Re\rho$.

From the formula for $F(s)$ from Step 2, the function
 has a pole at $\rho+1$ with
residue $r(\rho)=-2/\rho\neq 0$ in the half plane
$\sigma>1+B-\eps>3/2$, and we conclude that
\[
 1+B-\eps\leq\inf\{\sigma_0\ge\frac{3}{2}\mid
F(s)-\frac{1}{s-2}\text{ is analytic on }
  \sigma>\sigma_0\}.
\]
Letting $\eps\to 0$, we obtain  the desired goal.

Step 4.
Now let $D(x)=\sum_{n\leq x} G(n) -\frac{x^{2}}{2}$,
and we have $D(x)\ll x^{3/2+\eps}$ from \eqref{eq:ff}. 
Hence, as in the proof of \eqref{eq:fs},
\[
F(s)-\frac{1}{s-2}
=
s\int_1^{\infty}D(u)u^{-s-1}\du+\frac{1}{2}
\]
for $\sigma>2$,
where the right-hand side gives an analytic function on $\sigma>3/2$ since
$D(u)\ll u^{3/2+\eps}$ from \eqref{eq:ff}.

Therefore, by \eqref{eq:FB} from Step 3,
we conclude that $B\leq 1/2$ provided that $B<1$, hence RH.

Step 5.
We now need to exclude the possibility that  \eqref{eq:ff} could imply $B=1$ 
(see \cite{BhRu}).
For this, let $|z|<1$ and consider the power series
\[
   f(z)=\sum_{n\geq 1} \Lambda(n)z^{n} \text{ and }
   f^{2}(z)=\sum_{n\geq 1} G(n)z^{n},
\]
so that 
\[
  \frac{1}{1-z}f^{2}(z)=\sum_{n\geq 1} S(n)z^{n},
\] 
again with the summatory function
$S(x)=\sum_{m\leq x}G(m)$ of the Goldbach function $G(m)$.
Then
\[
    \frac{1}{1-z}f^{2}(z)=\frac{1}{(1-z)^{3}}+O(N^{5/2+\eps})
\]
on the circle $|z|=e^{-1/N}$, which can be reformulated as
\[
    f^{2}(z)=\frac{1}{(1-z)^{2}}+O(|1-z|N^{5/2+\eps}).
\]

This yields an asymptotic formula
on the major arc $|1-z|\leq cN^{-C/3}$.
Taking the complex {\emph square root} yields 
\[
   f(z)= \pm \frac{1}{1-z}+O(|1-z|^{2}N^{5/2+\eps}).
\]
Due to continuity and non-negativity of the coefficients of $f(z)$,
we have the plus sign throughout the whole major arc.

\newcommand{\dz}{\,\textup{d}z}
Now by Cauchy's integral formula, we obtain
\[
\psi(N)
=
\frac{1}{2\pi i} \int_{|z|=R} f(z) K(z)\dz,\quad
N
=
\int_{|z|=R}\frac{1}{1-z}K(z)\dz
\]
for the kernel
\[
K(z)=z^{-N-1}+z^{-N}+\dots+z^{-2}=z^{-N-1}\frac{1-z^{N}}{1-z}.
\]

The contribution of $f(z)$ 
to this integral is $O(N^{5/6})$ on the
major arc, and only 
$O(N^{11/12+\eps})$
for the minor arc (which needs a little care to
prove).

Comparing this with the explicit formula for $\psi(N)$, we conclude
that $B<11/12<1$.

\section{
Proof of Lemma 
on the number of exceptions in progressions}
\label{app:lem1}
For a positive integer $N$, for $a_{1},\dots,a_{N}\in \C$ 
and a residue $h$ mod $q$ denote
\[
   Z:=\sum_{n\leq N}a_{n} \text{ and } Z(q,h):=\sum_{\substack{n\leq
       N\\n\equiv h\ (q)}}a_{n}.
\]

We start with the proof of the following
\begin{theorem}
  \label{t2}
 For any real $H>0$ and $Q>1$ we have
   \begin{align*}
      \sum_{Q<m\leq 2Q} m &\max_{h\text{ \emph{mod} }m} 
       |Z(m,h)|^{2} \\ &\leq (N^{2}+Q^{2})
      \frac{\log Q}{H} \max_{n\leq N} |a_{n}|^{2}
      + (N+Q^{2})H\log Q \sum_{n\leq N}|a_{n}|^{2}.
   \end{align*}
\end{theorem}

\begin{proof}
We write $m\sim Q$ for $Q<m\leq 2Q$.
Split the left hand side of the theorem into
$E_{1}+E_{2}$ with
\begin{equation*}
  E_{1}:= \sum_{\substack{m\sim Q\\\tau(m)>H}} 
    m \max_{h \text{ mod } m}  |Z(m,h)|^{2}
\end{equation*}
and
\begin{equation*}
  E_{2}:= \sum_{\substack{m\sim Q\\\tau(m)\leq H}}
    m \max_{h \text{ mod } m}  |Z(m,h)|^{2},
\end{equation*}
where $\tau(m)$ denotes the number of divisors of $m$.

Consider first $E_{1}$. Let
\begin{equation*}
  A:=\#\{ m\sim Q; \;\tau(m) > H \},
\end{equation*}
then
\begin{equation*}
   A H < 
  \sum_{\substack{m\sim Q \\ \tau(m)> H}} \tau(m) 
  \leq\sum_{m\leq 2Q} \tau(m) \ll  Q\log Q,
\end{equation*}
so
\begin{equation*}
  A\ll \frac{Q\log Q}{H}.
\end{equation*}
Since $Z(m,h)\ll(\frac{N}{m}+1)\max_{n\leq N} |a_{n}|$ we get
\begin{align*}
  E_{1}&\ll \sum_{\substack{m \sim Q \\ \tau(m)>H }} 
      m \max_{h} |Z(m,h)|^{2} \ll 
      \sum_{\substack{m \sim Q \\ \tau(m)>H }} m\biggl(\frac{N^{2}}{m^{2}}
      +1\biggr) \max_{n\leq N} |a_{n}|^{2}  \\ &\ll
      A \biggl(\frac{N^{2}}{Q} +Q\biggr)\max_{n\leq N} |a_{n}|^{2}
      \ll \biggl(\frac{N^{2}}{H} +\frac{Q^{2}}{H}\biggr) \log Q
      \max_{n\leq N} |a_{n}|^{2}.
\end{align*}
This is the first summand on the right hand side of Theorem \ref{t2}.

\zeile
Now we look at $E_{2}$. From Theorem \ref{t1} below we have
\[
  E_{2}=\sum_{\substack{m\sim Q\\ \tau(m)\leq H}} m
   \max_{0< h\leq m} |Z(m,h)|^{2} 
   \leq \sum_{\substack{d\leq 2Q}} M'_{d} \sum_{\substack{ 0<b\leq d \\
     (b,d)=1}} \Big| T\Big( \frac{b}{d} \Big)\Big|^{2}.
\]
Now we estimate 
\[
   M'_{d}= \sum_{\substack{m\sim Q,d|m\\\tau(m)\leq H}}
   \frac{\tau(m)d}{m}\ll H\log Q
\]
and an application of the {\em large sieve inequality} yields
\[
   E_{2}\ll H \log Q\; (N+Q^{2})\sum_{n\leq N}|a_{n}|^{2},
\]
which is the second term on the right hand side of Theorem 
\ref{t2}. 
\end{proof}

We use the following result.
\begin{theorem}
\label{t1}
  We have the estimates
  \[
      \sum_{m\in\mathcal{M}} m \max_{h \text{ \emph{mod} }m} \Big|
      Z(m,h)-\frac{Z}{m}\Big|^{2} \leq \sum_{d=2}^{\infty} M'_{d}
      \sum_{\substack{ 0<b<d \\ (b,d)=1}} 
      \Big|T\Big(\frac{b}{d}\Big)\Big|^{2}
  \] and
  \[
     \sum_{m\in\mathcal{M}} m \max_{h \text{ \emph{mod} }m} 
      |Z(m,h)|^{2} \leq \sum_{d=1}^{\infty} M'_{d}
      \sum_{\substack{ 0<b\leq d \\ (b,d)=1}} 
      \Big|T\Big(\frac{b}{d}\Big)\Big|^{2}
  \]
  with
  \[ 
     M'_{d} := \sum_{t, td\in\mathcal{M}} \frac{\tau(dt)}{t}.
  \]
\end{theorem}
Here $\tau$ denotes the divisor function. 
Note that $M'_{d}\ll \tau(d)\log^{2}Q$ if
$\mathcal{M}\subseteq\{1,\dots,\lfloor Q\rfloor\}$ for a real number 
$Q>1$.

\begin{proof}
   Let 
\[
   f_{h}(m) := \sum_{d|m} \mu(d) \frac{m}{d} Z\Bigl(\frac{m}{d},h\Bigr),
\]
then we have by M\"obius inversion
\begin{align*}
  &\sum_{m} m \max_{h} |Z(m,h)-Z/m|^{2} =  \sum_{m} \frac{1}{m}
  \max_{h} |mZ(m,h)-Z|^{2} \\&= \sum_{m} \frac{1}{m} \max_{h} 
   \Big|\sum_{d|m} f_{h}(d)-Z\Big|^{2} = \sum_{m} \frac{1}{m}
   \max_{h}\Big|\sum_{\substack{ d|m\\d\neq 1}} f_{h}(d)\Big|^{2}\\
   &\leq \sum_{m} \frac{\tau(m)}{m} \sum_{\substack{ d|m\\d\neq 1}}
   \max_{h}|f_{h}(d)|^{2}.
\end{align*}
Now we note that $f_{h}(d)$ is $d$-periodic in $h$ for $d|m$,
since  $Z(t,h+d)=Z(t,h)$ for $t|d$, so
\begin{multline*}
  f_{h+dl}(d)=\sum_{t|d}\mu(t) \frac{d}{t}
  Z\Bigl(\frac{d}{t},h+dl\Bigr) = \sum_{t|d}\mu(t) \frac{d}{t}
  Z\Bigl(\frac{d}{t},h\Bigr) = f_{h}(d)\\ \text{ for all } l\in\Z,
\end{multline*}
therefore the maximum remains the same if taken only over $h$ with
$0< h\leq d$. We estimate this maximum by $\sum_{0< h\leq d}$, 
therefore an upper estimate for $\max_{0< h\leq q}|f_{h}(d)|^{2}$ is
\[
\sum_{h=1}^{d} |f_{h}(d)|^{2} = d\sum_{\substack{0\leq b \leq d\\(b,d)=1}}
\Big|T\Big(\frac{b}{d}\Big)\Big|^{2}
\]
by Montgomery's formula \cite{mf}, namely
\begin{equation*}
q\sum_{h=1}^{q}\Big| \sum_{d|q} \frac{\mu(d)}{d}
   Z\Big(\frac{q}{d},h\Big) \Big| ^{2}
   = \sum_{\substack{ 1\leq a\leq q \\ (a,q)=1}}
   \Big|T\Big(\frac{a}{q}\Big)\Big|^{2}
\end{equation*}
for the exponential sum 
\[
   T(\alpha):=\sum_{n\leq N} a_{n} e(\alpha n).
\]

So we get
\[
\sum_{m} m \max_{h} |Z(m,h)-Z/m|^{2} \leq \sum_{d=2}^{\infty} M'_{d}
\sum_{\substack{ 0<b<d \\ (b,d)=1 }} \Big|T\Big(\frac{b}{d}\Big)\Big|^{2}
\] 
with
\[
  M'_{d} := \sum_{m\in\mathcal{M}, d|m} \frac{\tau(m)d}{m} =
          \sum_{t, dt\in\mathcal{M}} \frac{\tau(td)}{t}.
\]
This shows the first inequality of Theorem \ref{t1}. 

The only change for the proof of the second inequality when replacing
$|Z(m,h)-Z/m|^{2}$ by $|Z(m,h)|^{2}$  is to include
the summand for $d=1$ in the sums over $d$. 
\end{proof}

Now we give the proof of Lemma \ref{lem:l}.
\begin{proof}
   Let $Q\leq x^{1/2}$ and $N=\lfloor x\rfloor$. 
  Then Theorem \ref{t2}, used with the indicator
   function for $B$ as sequence $(a_{n})_{n\leq N}$, shows
\begin{equation*}
    \sum_{q\sim Q} \max_{h\text{ mod }q} B_{h,q}(x) \ll
    (x^{2}H^{-1}+HxB(x))^{1/2}\log ^{1/2}Q.
\end{equation*}
The optimal choice for $H$ is $H=(x/B(x))^{1/2}$, and therefore
we have 
\begin{equation}
\label{eq:ea2}
\sum_{q\sim Q} \max_{h\text{ mod }q} B_{h,q}(x) 
   \ll x^{3/4}B(x)^{1/4}\log ^{1/2}Q.
\end{equation}

So if $B(x)$ is small, we expect  $B_{h,q}(x)$ to be
small too.

Consider therefore for $C>0$ the number of exceptional modulii
\[
   \mathcal{M}_{Q}:=\{q\sim Q;\ B_{h,q}(x)>\frac{x}{q\log ^{C}x}
   \text{ for any } h \text{ mod }q\}.
\]
It follows that 
\begin{align*}
   &\#\mathcal{M}_{Q}\cdot\frac{x}{Q\log ^{C}x} \ll
   \sum_{q\in\mathcal{M}_{Q}} \max_{h} B_{h,q}(x) \\ &\qquad\ll x^{3/4}
   B(x)^{1/4} \log ^{1/2}Q \ll \frac{x}{\log ^{2C+1}x},
\end{align*}
so $\#\mathcal{M}_{Q}\ll \frac{Q}{\log ^{C+1}x}$.

If we split $[1,x^{1/2}]$ into $\ll \log x$ many dyadic intervals,
we get
\[
   \#\{q\leq x^{1/2};\; B_{h,q}(x)>\frac{x}{q\log ^{C}x} \text{ for
     any }h\text{ mod }q \} \ll \frac{x^{1/2}}{\log ^{C}x}.
\]
This proves the Lemma.
\end{proof}



\begin{thebibliography}{50}
%
\bibitem{CB} 
C. Bauer, Goldbach's conjecture in APs:
number and size of exceptional prime moduli, {\it Arch. Math.} {\bf 108} (2017), 
159--172
%
\bibitem{B_SP} G. Bhowmik and J.-C. Schlage-Puchta,
Mean representation number of integers as the sum of primes,
{\it Nagoya Math. J.}  {\bf 200} (2010), 27--33.
%
\bibitem{B_SP11} G. Bhowmik and J.-C. Schlage-Puchta,
Meromorphic continuation of the Goldbach generating function,
{\it Funct. Approx. Comment. Math.}  {\bf 45} (2011), 43--53.
%
\bibitem{BHMS}
G. Bhowmik, K. Halupczok,  K. Matsumoto and Y. Suzuki,
Goldbach Representations in Arithmetic Progressions and zeros of Dirichlet $L$-functions,
{\it Mathematika} [ Published online: 24 August (2018)],  57--97.
%
\bibitem{BhRu} G. Bhowmik and I.Z. Ruzsa, Average Goldbach and the Quasi-Riemann Hypothesis, 
{\it Analysis Mathematica} {\bf 44(1)} (2018), 51--56.
%
\bibitem{vdc}
J. G. van der Corput, Sur l'hypoth\`ese de Goldbach pour pr\`esque tous
les nombres pairs, {\it Acta Arith.}  {\bf 2} (1937), 266--290.
%
\bibitem{chu}
G. Chudakov, On the Goldbach problem, {\it C. R. Acad. Sci. URSS,
(2)}{\bf 17} (1937), 335--338.
%
\bibitem{Desboves}
A. Desboves, Sur un th\'eor\`eme de Legendre  et son application \`a la recherche 
de limites qui comprennent entre elles des nombres premiers, {\it Nouv. Ann. Math.} {\bf 14} (1855),
81--295.
%
\bibitem{EgaMatsu} S. Egami and K. Matsumoto; 
Number theory, 1--23, Ser. Number Theory Appl., 2, World Sci. Publ.,
Hackensack, NJ, 2007.
%
\bibitem{est}
 T. Estermann, On Goldbach's problem: Proof that almost all even
 positive integers are sums of 
two primes, {\it Proc. London Math. Soc.(2)} {\bf 44} (1938), 307--314.
%
\bibitem{fei}
J. H. Fei, An application of the Hardy--Littlewood conjecture,
{\it J. Number Theory} {\bf 168} (2016), 39--44. 
%
\bibitem{fuj}
A. Fujii,
An additive problem of prime numbers,
{\it Acta Arith.} {\bf 58} (1991), 173--179.
%
\bibitem{fuj2}
A. Fujii,
An additive problem of prime numbers II,
{\it Proc. Japan Acad.} {\bf 67}, Ser. A (1991), Number 7, 248--252.
%
\bibitem{Gio}
G. Giordano, On the irregularity of the distribution of the sums of pairs of odd primes, {\it Int. J. of
Math. and Math. Sc.} {\bf 30:6} (2002), 377--381.
%
\bibitem{Gol}
 D.A. Goldston, On Hardy and Littlewood's contribution to
  the Goldbach conjecture.  Proceedings of the Amalfi Conference on
  Analytic Number Theory (Maiori, 1989),  115--155, Univ. Salerno,
  Salerno, 1992.
%
\bibitem{GY}
D.A. Goldston and L.Yang, The Average Number of Goldbach Representations, in
"Prime Numbers and Representation Theory", Lecture Series of Modern Number Theory, Vol. 2, 2017. 
%
\bibitem{Gra}
A. Granville,
Refinements of Goldbach's conjecture, and the generalized Riemann hypothesis,
{\it Funct. Approx. Comment. Math.} {\bf 37} (2007), 159--173;
Corrigendum, ibid. {\bf 38} (2008), 235--237.
%
\bibitem{HL}
G. H. Hardy and J. E. Littlewood,
Some problems of "partitio numerorum" (V): A further contribution to
the study of Goldbach's 
problem,
{\it Proc. London Math. Soc. (2)} {\bf 22} (1924), 46--56.
%
\bibitem{HR}
G. H. Hardy and S. Ramanujan, Asymptotic Formul\ae\  in Combinatory Analysis,
{\it Proc. London Math.Soc.}  {\bf 17} (1918), 75--115.

%
\bibitem{Land}
E. Landau, \"Uber die zahlentheoretische Funktion $\phi(n)$ und ihre Beziehung zum Goldbachschen Satz, 
{\it G\"ottinger Nachrichten} (1900), 177--186.
%
\bibitem{Languasco}
A. Languasco, Applications of some exponential sums on prime powers: a survey,
{\it Riv. Mat. Univ. Parma} {\bf 7} (2016), 19--37.
%
\bibitem{LaZa}
A. Languasco and A. Zaccagnini, The number of Goldbach representations of an integer, 
{\it Proc. Amer. Math. Soc.} {\bf 140} (2012), 795--804.
%

\bibitem{LZ} M.-C. Liu and T. Zhan, The Goldbach problem with primes
  in arithmetic progressions, in "Analytic Number Theory",
  Y. Motohashi (ed.), 
London Math. Soc. Lecture Note Ser. {\bf 247}, Cambridge Univ. Press,
1997, 227--251. 
%
\bibitem{Lu}
W. C. Lu, Exceptional set of Goldbach number,
{\it J. Number Theory} {\bf 130} (2010), no. 10, 2359--2392. 
%
\bibitem{mf}
H.L. Montgomery, A note on the large sieve, {\it J. London
  Math. Soc.} {\bf 43} (1968),  93--98. 
%
\bibitem{MV}
H.L. Montgomery and R.C. Vaughan, 
The exceptional set in Goldbach's problem,
Collection of articles in memory of Juri{i} Vladimirovi{c} Linnik.
{\it Acta Arith.} {\bf 27} (1975), 353--370.
%
\bibitem{Moz}
C.J. Mozzochi,  A Comparison of Sufficiency Condtions for the Goldbach and the Twin Primes Conjectures. 
{\it Advances in Pure Mathematics} {\bf 4} 2017, 157--170.
%
\bibitem{OHP}
 T. Oliveira e Silva, S. Herzog, S. Pardi, Empirical verification of the even Goldbach conjecture and computation of prime
gaps up to $4 \times 10^{18}$ {\it Math. Comp.}  {\bf 83} (2014), 2033--2060.
%
\bibitem{pin}
J. Pintz, A new explicit formula in the additive theory of primes with applications II. The exceptional set in Goldbach's problem.
 arXiv:1804.09084v2 [math.NT] 
%
\bibitem{Pra}
K. Prachar, On integers n having many representations as sum of two primes, 
{\it J. London Math. Soc.} {\bf 29} (1954), 347--350.
%
\bibitem{Ruepp}
F. R\"uppel, Convolutions of the von Mangoldt function over
residue classes. {\it {\v S}iauliai Math. Semin.} {\bf 7(15)} (2012), 135--156.
%
\bibitem{Suzuki}
Y. Suzuki,
A mean value of the representation function for the sum of two primes in arithmetic progressions,
{\it Int. J. Number Theory} {\bf 13 (4)} (2017), 977--990.
%
\bibitem{VauOver}
R. C. Vaughan, Goldbach's  Conjectures: A Historical
Perspective, in "Open Problems in Mathematics", Springer, 
 2016, 479--520.
%
\bibitem{Vino}
I. M. Vinogradov, Representation of an odd number as the sum of three
primes, {\it Dokl. Akad. Nauk SSSR}  {\bf 16} (1937), 179--195.
\end{thebibliography}
\end{document}